\documentclass[11pt,twoside]{amsart}

\usepackage[margin=1.5in]{geometry}
\usepackage{fancyhdr}

\usepackage{amsmath, amssymb, amsthm, amscd}
\usepackage{mathtools}
\usepackage{mathrsfs}
\usepackage{enumitem}
\usepackage{comment}
\usepackage{tikz}
\usepackage{tikz-cd}
\usepackage{thmtools}
\usepackage[backref=page]{hyperref}
\usepackage[dvipsnames]{xcolor}
\hypersetup{colorlinks=true,citecolor=NavyBlue,linkcolor=BrickRed,urlcolor=Orange}
\usepackage[nameinlink,capitalize]{cleveref}

\theoremstyle{plain}
\newtheorem{theorem}{Theorem}[section]
\newtheorem{proposition}[theorem]{Proposition}
\newtheorem{lemma}[theorem]{Lemma}
\newtheorem{corollary}[theorem]{Corollary}

\theoremstyle{definition}

\newtheorem{example}[theorem]{Example}

\theoremstyle{remark}

\newtheorem*{remark}{Remark}

\newcommand{\theoremref}[1]{\hyperref[#1]{Theorem~\ref*{#1}}}
\newcommand{\lemmaref}[1]{\hyperref[#1]{Lemma~\ref*{#1}}}
\newcommand{\definitionref}[1]{\hyperref[#1]{Definition~\ref*{#1}}}
\newcommand{\propositionref}[1]{\hyperref[#1]{Proposition~\ref*{#1}}}
\newcommand{\conjectureref}[1]{\hyperref[#1]{Conjecture~\ref*{#1}}}
\newcommand{\corollaryref}[1]{\hyperref[#1]{Corollary~\ref*{#1}}}
\newcommand{\exampleref}[1]{\hyperref[#1]{Example~\ref*{#1}}}

\newcommand{\R}{\mathbb{R}}
\newcommand{\C}{\mathbb{C}}
\newcommand{\Z}{\mathbb{Z}}

\newcommand{\F}{\mathscr{F}}
\newcommand{\D}{\mathscr{D}}

\newcommand{\V}{\mathcal{V}}

\newcommand{\strucsheaf}{\mathcal{O}}
\newcommand{\proj}{\mathbb{P}}

\DeclareMathOperator{\sheafhom}{\mathcal{H\kern -3pt o\kern -2pt m\kern -1pt}}

\newcommand{\IM}{\operatorname{Im}}

\newcommand{\defeq}{\mathrel{\mathop:}=}

\newcommand{\gr}{\operatorname{gr}}

\newcommand\sbullet[1][.5]{\mathbin{\vcenter{\hbox{\scalebox{#1}{$\bullet$}}}}}


\title[Nefness of lowest piece]{On nefness of the lowest piece of Hodge modules}

\author[Ze Yun]{Ze Yun}
\address{Department of Mathematics, Stony Brook University}
\email{ze.yun@stonybrook.edu}

\date{\today}

\begin{document}

\begin{abstract}
    % We give degree lower bounds for quotient line bundles of the lowest piece of a Hodge module induced by a complex variation of Hodge structures outside a simple normal crossing divisor, beyond the unipotent variation case. This note aims to explain the failure of nefness when the monodromies at infinity are not unipotent. The lower bounds depend on local monodromies at infinity and intersection numbers with the boundary divisors. In particular it recovers Kawamata's semi-positivity theorem for unipotent variations. The proof is algebraic via a vanishing theorem for twisted Hodge modules. We also give geometric examples to show that the lower bound can be achieved. 
	Let \(X\) be a smooth projective complex variety, \(D=\sum D_i\) a simple normal-crossing divisor, and \(V\) a polarizable complex variation of Hodge structure on \(X\setminus D\). For the first nonzero Hodge piece of its intermediate Hodge module extension, we give lower bounds for the degrees of quotient line bundles along curves in terms of residue eigenvalues and the degrees of normal bundles of the boundary divisors. The proof uses nearby cycles of twisted Hodge modules along divisors to study restrictions to boundary strata and applies twisted Kodaira-Saito vanishing on curves. In particular it recovers Kawamata's semi-positivity theorem for unipotent variations. We also give geometric examples to show that the lower bound can be achieved. 
\end{abstract}

\maketitle
% \newpage
% \tableofcontents

\section{Introduction}
	\indent \indent The first nonzero step of the Hodge filtration of a variation of Hodge structures, equipped with its Hodge metric, has semi-positive curvature by the work of Griffiths. A common example is given by $R^j f_*\omega_{X/S}$ for a smooth projective surjective map $f: X\rightarrow S$, and positivity of higher direct images of canonical bundles plays an important role in algebraic geometry. Kollár studied them in \cite{HigherdirectimagesI} \cite{HigherdirectimagesII}, and showed many good properties of them (torsion-freeness, vanishing theorem, etc.) even when the family is not smooth. He conjectured that a variation of Hodge structure should admit a sheaf analogous to $R^jf_*\omega_{X/S}$ as above and share the same kinds of good properties as $R^jf_*\omega_{X/S}$. Saito gave an answer to Kollár's conjecture using the theory of mixed Hodge modules \cite{SaitoKollarconjecture}. The analogous sheaf for a variation of Hodge structures $\mathcal{V}$ is the lowest piece of the Hodge filtration of the Hodge module induced by $\mathcal{V}$. An example of such a sheaf is given by $R^jf_*\omega_{X/S}$, where $f:X\rightarrow S$ is a surjective proper morphism between smooth projective varieties. In this paper, we study the nefness of these sheaves.
    
	On curves, the lowest piece of a Hodge module is always a nef vector bundle, which is proved using asymptotic behavior of Hodge metrics in \cite{fujitasemipositivity}, \cite{zucker}, and \cite{peters}. In the case of an SNC discriminant with unipotent monodromy along the discriminant divisor, Kawamata \cite[Theorem 5]{Kawamata1981} proved the lowest piece of the Hodge filtration of the canonical extension of a variation of Hodge structures across the divisor is a nef vector bundle. In their proofs, the analytic results in \cite{cks} about the asymptotics of the Hodge metric on the canonical extension and the curvature properties of the Hodge bundle are essential. Later, Fujino-Fujisawa and Fujino-Fujisawa-Saito \cite{fujinovmhs}, \cite{MR3167580}, \cite{semipo} had further results in this direction. The positivity of the lowest nonzero Hodge piece of a Hodge module has also been studied from several complementary perspectives. Schnell proved its weak positivity \cite{weakpositivity}, while Schnell and Yang showed that its Hodge metric has semipositive curvature and satisfies the minimal-extension property \cite{singularmetric}. The purpose of this paper is to quantify the resulting failure of nefness along boundary strata in terms of the residue eigenvalues and the normal bundles of the boundary components, from the perspective of Saito's vanishing theorem and its twisted version.

	In dimensions greater than or equal to two, the example in the work of Fujino and Fujisawa \cite[example 7.9]{fujinovmhs} shows nefness fails in general. We quantify this failure of nefness by giving lower bounds for the degrees of quotient line bundles of the lowest Hodge piece pulled back to curves. The proof of the following theorem is algebraic by using Saito's vanishing theorem for twisted Hodge modules on curves and does not need an analysis of the asymptotics of the Hodge metric. 
	
	Let $X$ be a compact complex manifold, let $D=\sum_{i} D_i$ be a simple normal crossing divisor on $X$, and let $\mathcal{V} $ be a polarized complex variation of Hodge structures on $X\backslash D$. We know from the structure theorem of polarized Hodge modules \cite[Theorem 16.2.1]{SSMHMproject} that any Hodge module with strict support on $X$ is induced by a variation of Hodge structures on a Zariski open subset. By Saito's theory, there is a Hodge module $M$ with strict support on $X$, which extends the VHS $\mathcal{V}$. In fact $M=\D_X\cdot \mathcal{V}^{>-1}$ is generated by the Deligne canonical extension \cite{overview}. It carries a good filtration called the Hodge filtration, which agrees with the Hodge filtration of $\mathcal{V}$ on $X\backslash D$. We let $F_p M$ denote the first nonzero step of the Hodge filtration of $M$, i.e., $F_{p-1}M=0$ and $F_p M\neq 0$. The lowest piece $F_p M$ is equal to $j_* F_p \mathcal{V} \cap \mathcal{V}^{>-1}$ \cite[Exercise 11.2]{overview}, where $j: X\backslash D\rightarrow X$ is the inclusion. In this case, $F_p M$ is locally free by the nilpotent orbit theorem \cite{schmid} \cite{cks} \cite{sabbah2022degeneratingcomplexvariationshodge}.
	
	% Let $M$ be the (left) Hodge module with strict support on $X$ induced by $\mathcal{V}$, and $F_p M$ denote the first nonzero step of the Hodge filtration of $M$, i.e., $F_{p-1}M=0$ and $F_p M\neq 0$. $F_p M$ is locally free in this situation by the nilpotent orbit theorem \cite{schmid} \cite{cks} \cite{sabbah2022degeneratingcomplexvariationshodge}.

	For a fixed $i$, the real numbers (rational in the case of quasi-unipotent monodromies) $-1<\beta_{ij}\leq 0$ are eigenvalues of residues of the Deligne canonical extension of $\mathcal{V}$ along $D_i$. We introduce the following notation: let $I=\{1,\cdots,k\}$,
	$$\Sigma_I= \{ \beta=(\beta_{i})_{i\in I}: \text{$\beta_{i}$ is an eigenvalue of the residue along $D_i$ of the canonical extension of $\mathcal{V}$}\}.$$
	
		\begin{restatable}{theorem}{maintheorem}\label{thm:main}
		Let $C$ be a smooth projective curve, and let $f: C\rightarrow X$ be a morphism. Then the following hold.
		\begin{itemize}
			\item[(a)] If $\IM(C)\nsubseteq D$, then $f^* F_p M$ is a nef vector bundle on $C$.

			\item[(b)] If $\IM(C)\subset D_{(k)}\defeq D_1\cap \cdots \cap D_k$, but the image of $C$ is not contained in the intersection of $k+1$ distinct boundary components, and $f^* F_p M\rightarrow P$ is a quotient line bundle, then 
			$$\deg P \geq \min_{\beta\in \Sigma_I} \big\{ \sum_{i=1}^k (- \beta_{i}\cdot \deg f^* \strucsheaf_X(D_i)) \big\}.$$
	
		\end{itemize}

		\end{restatable}

% We now prove \theoremref{thm:main}.

% % 重新显示 Theorem 1.1，不产生新编号
% \maintheorem*

    When the monodromies are unipotent, all $\beta_{ij}$ are equal to zero, so $F_p M$ is a nef vector bundle on $X$. Therefore our result is a generalization of the semi-positivity theorem of Kawamata \cite{Kawamata1981}, see also \cite{fujinovmhs}, \cite{MR3167580}, \cite{semipo} for discussions. The theorem also implies $F_p M$ is nef whenever boundary divisors are nef. Our main result also explains some numeric failure of nefness, for example the one given in \cite[example 7.9]{fujinovmhs}. See section 4 for a detailed exposition and section 6 for some other examples in higher dimensions.
    
    \begin{corollary}[\cite{Kawamata1981}]
        In the same setting as above, if the monodromies at infinity are unipotent then the lowest piece $F_pM$ of the Hodge module is nef.
    \end{corollary}

    \begin{corollary}
        In the same setting as above, if each component $D_i$ of $D$ is a nef divisor (for example $X$ is a compact complex torus), then the lowest piece $F_pM$ of the Hodge module is nef.
    \end{corollary}

    In \cite{pps2017}, when $X$ is a complex torus, they proved via generic vanishing theory that $F_p M$ is nef, which agrees with our result since any effective divisor on a complex torus is nef. In fact, they proved a stronger result: the lowest piece on a complex torus carries a smooth metric with semi-positive curvature, hence is nef.

    Another outcome of our method is the following proposition. A general complete intersection curve is not contained inside the boundary divisor and nefness holds true for such curves, by the first part of the main theorem. Our result says nefness still holds even if it's a complete intersection lying inside the boundary divisor, which seems very surprising. Note that we also don't assume $D$ is simple normal crossing in the following proposition.

\begin{proposition}
	Let $X$ be a smooth projective variety of dimension $n$, and let $M$ be a Hodge module induced by a variation of Hodge structures on a Zariski open subset of $X$. Let $f:C\rightarrow X$ be a map from a smooth projective curve $C$. If the image of $C$ is a smooth complete intersection curve of ample divisors, $f^*F_p M$ is a nef coherent sheaf on $C$. 
\end{proposition}

	As suggested by the main result, the failure of nefness comes from the locus where the Hodge module is not smooth, which is also shown by the counter-examples in \cite[example 7.9]{fujinovmhs}. For curves intersecting the open set underlying the variations of Hodge structures, it can be derived from an argument in \cite[Lemma 2.6]{bakker2023hodgetheoreticproofhwangstheorem} that the pullback of the lowest piece is still a nef vector bundle on the curve. For curves contained in the boundary divisor, we need to restrict the lowest piece to $D_i$ and hope to do some kind of induction. The problem is directly restricting $F_p M$ to $D_i$ doesn't give an object that directly comes from a Hodge module on $D_i$. Our method relies on the notion of twisted Hodge modules \cite{schnell2024highermultiplierideals} \cite{davis2025unitaryrepresentationsrealgroups}, which naturally arises when we specialize a Hodge module to divisors. In principle, when the monodromy at infinity is quasi-unipotent, our result could be proved using a branched cover. But our method is more direct and works for complex variations of Hodge structures which need not be quasi-unipotent, and is a nice application of twisted Hodge modules. We also have some geometric examples where the lower bounds are achieved.
	
	\textbf{Notation, convention, and outline of the paper.}
	In this paper, all the complex variations of Hodge structures are polarized, and they are just abbreviated as VHS. Canonical extension of a VHS will mean Deligne's canonical extension of a flat bundle to a logarithmic flat connection with eigenvalues of residues in the interval $(-1,0]$. Simple normal crossing divisors are abbreviated as SNC divisors. A coherent sheaf $\F$ is nef if its relative $\strucsheaf(1)$ on $\proj(\F)$ is a nef line bundle. Equivalently, $\F$ is nef if for every map $f:C\rightarrow X$ from a smooth projective curve, and every quotient line bundle $f^* \F \rightarrow L$, $\deg L\geq 0$. $M$ always stands for a left $\D$-module, whereas the letter $N$ stands for a right $\D$-module; the lowest piece of Hodge filtration of a (left) Hodge module $M$ is always denoted by $F_p M$; the corresponding lowest piece for the associated right $\D$-module $N=M\otimes K_X$ is denoted by $F_p N$, which is equal to $F_p M\otimes K_X$, where $X$ is the ambient space. Our index of Hodge filtration of $N$ is not standard, but we stick to this way for simplicity, so the index does not cause confusion when we need to switch between left and right $\D$-modules.

    In section 2, we give a brief introduction to twisted Hodge modules and explain their relation to our problem. In section 3, we argue that on curves, nefness is due to Saito's vanishing theorem. Section 4 is devoted to an explanation of a counter-example from \cite{fujinovmhs}, using our main result. In section 5, we prove the main theorem and give some corollaries. The last section gives some examples to show that our lower bound can be achieved.
%\newpage

\section*{Acknowledgment}

I would like to thank my advisor Christian Schnell for suggesting this problem, and for his constant support and encouragement. Our weekly meetings, which always run overtime, are valuable to me. I also want to thank Zhengze Xin and Dingchang Zhou for various discussions on the problem. Special thanks to Mattew Huynh, Shuo Gao, and Yipeng Zhou for listening to my practice talk and offering helpful suggestions.

% An AI model is used for improving the language of this paper. No mathematical statements or proofs in this paper are produced by AI. The author takes the sole responsibility for the validity of results in this paper.

\section{Twisted Hodge modules and nearby cycles}
One key ingredient of our proof is the notion of twisted Hodge modules \cite{schnell2024highermultiplierideals} \cite{davis2025unitaryrepresentationsrealgroups}. In our situation, they arise naturally when taking nearby cycles $\gr^W\gr^V M$ along a divisor, where $V$ is the $V$-filtration, and $W$ is the monodromy filtration on $\gr^V M$. Let $L$ be a line bundle on a complex manifold $X$, $\alpha\in \R$. The ring of $\alpha L$-twisted differential operators is a quasi-coherent sheaf on $X$ which is locally isomorphic to the usual sheaf of differential operators,
$$\D_{X,\alpha L}|_{U} \cong \D_U[t\partial_t]/(t\partial_{t} -\alpha) \cong \D_U,$$
but with a different gluing to a sheaf on $X$. The transition map between differential operators under trivializations $\phi, \phi': L|_U\rightarrow U\times \C$ is
$$t\partial_t\rightarrow t'\partial_{t'},\quad \partial_j\mapsto \partial_j + \alpha \cdot g^{-1}\frac{\partial g}{\partial x_j},$$
where $\phi' \circ \phi^{-1}: U \times\mathbb{C} \rightarrow  U \times \mathbb{C}$ has the form $(x,t)\mapsto (x,g(x)t)$ for some invertible holomorphic function $g$ on $U$. 

In our case, if $D$ is an effective divisor, and $L=\strucsheaf_X(D)$, then one result in \cite{schnell2024highermultiplierideals} says twisted Hodge modules arise as nearby cycles of Hodge modules. Consider the $V$-filtration of $M$ along $D$, and let $W$ be the monodromy weight filtration on $\gr^V M$.

\begin{proposition}[Proposition 2.16 of \cite{schnell2024highermultiplierideals}]
	 Let $\left(\mathcal{M}, F_{\bullet} \mathcal{M}\right)$ be a polarized complex Hodge module on $X$. For any real number $\alpha \in[-1,0)$ and any integer $\ell \in \mathbb{Z}$, the object

		$$ \left(\operatorname{gr}_{\ell}^{W(N)} \operatorname{gr}_\alpha^V \mathcal{M}, F_{\bullet-1} \operatorname{gr}_{\ell}^{W(N)} \operatorname{gr}_\alpha^V \mathcal{M}\right)
		$$
		is a polarized $\alpha L$-twisted Hodge module on $X$. The object
		$$ \left(\operatorname{gr}_{\ell}^{W(N)} \operatorname{gr}_0^V \mathcal{M}, F_{\bullet} \operatorname{gr}_{\ell}^{W(N)} \operatorname{gr}_0^V \mathcal{M}\right)
		$$
		is a polarized complex Hodge module (without any twisting).
\end{proposition}

The relation of twisted Hodge modules with our problem is the following. Suppose the boundary divisor is $D=\sum D_i$ which is SNC. Consider the $V$-filtration on $D_1$. By \cite[exercise 11.2]{overview}, for a right $\D$-module which underlies a Hodge module, $F_p N= F_p V_{<0}N$. Then 
		$$F_p N|_{D_1} = \frac{F_p N}{t\cdot F_p N} = \frac{ F_p V_{<0}N}{t\cdot F_p V_{<0}N} = \frac{ F_p V_{<0}N}{F_pV_{<-1}N},$$
		hence $F_p N|_{D_1} $ is filtered by coherent subsheaves 
		$$\frac{F_pV_{\alpha_1}N}{F_p V_{<-1}N},\cdots,\frac{F_pV_{\alpha_k}N}{F_p V_{<-1}N},$$
		where $-1\leq \alpha_1<\cdots<\alpha_k<0$ are related to eigenvalues of residues along $D_1$. The nefness of $F_p N|_{D_1}$ is related to the nefness of $F_p \gr^V N$ and $F_p \gr^W\gr^V N$.

Now we explain another new ingredient of the proof. In the proof of the main theorem, we need to specialize a Hodge module to intersections of divisors several times. This raises the question: what's the meaning of taking nearby cycles for a (twisted) Hodge module. In fact, the definition of the $V$-filtration along a local analytic hypersurface given in \cite[Definition 2.14]{schnell2024highermultiplierideals} also applies to twisted Hodge modules: its defining properties are preserved by the transition formulas for differential operators under changes of trivialization of $L$. Therefore, the graded quotient of $V$-filtration and nearby cycles of an $\alpha L$-twisted Hodge module along another divisor still make sense, and should be understood as a Hodge module that is twisted twice. Therefore, we can iteratively take nearby cycles of a Hodge module along components of a simple normal crossing divisor, and what we get is an object that can be called a multiply twisted Hodge module. To compare with the terminology of \cite{davis2025unitaryrepresentationsrealgroups}, the thing that is called an $(\alpha_1 L_1,\cdots , \alpha_k L_k)$-twisted Hodge module in our paper is called a $\lambda$-twisted Hodge module in their setting, where $\lambda=(\alpha_1,\cdots,\alpha_k)\in \mathfrak{h}^*$, $\mathfrak{h}$ is the Lie algebra of the torus $(\C^*)^k$, and the $(\C^*)^k$ torsor is $\tilde{X}\defeq \Pi_{i=1}^k (L_i\backslash X)$. However, in general, the process of iteratively taking graded pieces of $V$-filtration and weight filtration does not commute when we change the order of doing nearby cycles along different effective divisors. For the problem we are considering, since we only study the lowest piece of Hodge filtration, and the boundary divisor is SNC in this paper, the effect of all the operations that we do to the lowest piece is in fact independent of the order. This will follow from general results about variations of Hodge structures \cite{schmid} \cite{cks}, and \lemmaref{lem:lowestpiecenearbycycle}.

By applying nearby cycles for twisted Hodge modules, we can reduce the dimensions of supports of our (twisted) Hodge modules. Let $L,Q$ be two line bundles. To illustrate, a ring of doubly twisted differential operators should locally look like
$$\D_{X,\alpha L,\beta Q}|_{U} \cong \D_U[t\partial_t,s\partial_s]/(t\partial_t -\alpha,s\partial_s-\beta) \cong \D_U,$$
their transitions under a simultaneous trivialization of $L$ and $Q$ should be 
$$t\partial_t\rightarrow t'\partial_{t'},\quad s\partial_s\mapsto s'\partial_{s'}, \quad \partial_j\mapsto \partial_j + \alpha \cdot g^{-1}\frac{\partial g}{\partial x_j}+ \beta\cdot h^{-1}\frac{\partial h}{\partial x_j}, $$
where $g$ and $h$ are transition functions for $L$ and $Q$. An $(\alpha L,\beta Q)$-twisted $\D$-module is a module over this ring.

\section{Lowest piece of twisted Hodge modules on curves}
\indent\indent First we prove the lowest piece of an untwisted Hodge module of strict support on a smooth projective curve $C$ is always a nef vector bundle, using Saito's vanishing theorem. This is a simple exercise in applying the vanishing theorem, but it shows how our argument works. The advantage of using vanishing theorems is that we don't need to do careful analysis of the growth of the Hodge metric as in \cite{fujitasemipositivity} \cite{zucker} \cite{peters}. 

\begin{theorem}\label{nefness on curves}
	Let $M$ be a Hodge module of strict support $C$, where $C$ is a smooth projective curve, and let $F_p M$ be its lowest piece. Then $F_p M$ is a nef vector bundle.
\end{theorem}
\begin{proof}
	 Let $F_p M\rightarrow P$ be a quotient line bundle on $C$. Then it can be viewed as a nonzero element in $H^0(C, F_pM^*\otimes P)\cong H^1(C,F_p M\otimes K_C\otimes P^{-1})^*$. By Saito's vanishing theorem, if $\deg P<0$, then $H^1(C,F_p M\otimes K_C\otimes P^{-1})=0$, a contradiction. Therefore $\deg P\geq 0$. 
	 
	 In general if $g: D\rightarrow C$ is a map between smooth projective curves, we first pullback the generic variation of Hodge structures on $C$ to $D$, and extend it as a Hodge module $M_D$ on $D$. Then there exists a map $F_p M_D\rightarrow g^* F_p M$ which is generically an isomorphism \cite[Lemma 2.6]{bakker2023hodgetheoreticproofhwangstheorem}. By the discussion above, the quotient line bundles of $g^* F_p M$ will also have non-negative degrees. Therefore $F_p M$ is a nef vector bundle on $C$. 
\end{proof}

\begin{remark}
	$F_p M$ has a trivial quotient if and only if the local system of the complex variation of Hodge structures on a Zariski open set of $C$ has a trivial summand, such that the intersection of the trivial summand with the lowest piece is non-zero. This is implied by \cite[section 2]{bakker2023hodgetheoreticproofhwangstheorem}. 
	
	In other words, the lowest degree quotient is achieved only when the variation has a fixed part in the lowest piece (up to tensoring with a unitary rank one VHS).
\end{remark}

If we have a vanishing theorem for twisted Hodge modules, then we could also argue as in Theorem \ref{nefness on curves}, and the argument will also work for real twistings. This indeed works thanks to a generalization of Saito's vanishing theorem \cite[Theorem 1.7]{schnell2024highermultiplierideals} \cite[Theorem 4.2]{davis2025unitaryrepresentationsrealgroups}, and the above theorem is still true, even with $\R$-twistings (although with $\R$-twistings, the equality is very unlikely to be achieved). 

\begin{theorem}\label{thm:curvetwistingtheorem}
	Let $\alpha = (\alpha_1,\cdots,\alpha_k)\in \R^k$, and $M$ be an $(\alpha_1 L_1,\cdots,\alpha_k L_k)$-twisted Hodge module with strict support on $C$. For any surjective map of smooth projective curves $f:C_1\rightarrow C$, and any line bundle quotient $f^* F_p M\rightarrow P$, $$\deg P\geq \sum_{i=1}^k \alpha_i\cdot\deg f^* L_i .$$
\end{theorem}
\begin{proof}
	The proof is basically the same as the untwisted case. First let $F_p M\rightarrow P$ be a quotient line bundle on $C$. Suppose $-\deg P +\sum_{i=1}^k \alpha_i\cdot\deg L_i >0$. Then by the vanishing theorem \cite[Theorem 4.2]{davis2025unitaryrepresentationsrealgroups}, $H^1(C,F_p M\otimes P^{-1}\otimes K_C)=0$, so there can't be such surjections $F_p M\rightarrow P$.  
	
	In general, given a map $f: C_1\rightarrow C$, we want to consider quotient line bundles $f^* F_p M\rightarrow P$. We can first pullback the generic twisted VHS to $C_1$, extend it as an $(\alpha_1 f^*L_1,\cdots,\alpha_k f^* L_k)$-twisted Hodge module $M'$ on $C_1$ \cite[Theorem 2.13]{schnell2024highermultiplierideals}, then there exists a map $F_p M'\rightarrow f^* F_p M$ which is generically an isomorphism as in the untwisted case (see \lemmaref{lem:twisted VHS generic pullback lowest piece}). The conclusion follows from the above paragraph.
\end{proof}

\section{A counter-example to nefness}
In this section, we explain the failure of nefness in the example of \cite[Example 7.9]{fujinovmhs} with our main theorem. 

Let $X$ be a compact complex manifold and $D=\sum_i D_i$ be an SNC divisor on $X$. $\mathcal{V} $ is a polarized variation of Hodge structures on $X\backslash D$. Recall the Deligne canonical extension with eigenvalues of residues in $(-1,0]$, denoted by $\mathcal{V}^{>-1}$, is a vector bundle with logarithmic connection, and eigenvalues of residues are real numbers in the interval $(-1,0]$. If we assume the monodromies of $\mathcal{V}$ along $D$ are quasi-unipotent, then the eigenvalues of residues are in fact rational numbers in $(-1,0]$. The Hodge module $M$ extending the VHS $\mathcal{V}$ is a $\D$-module $M=\D_X\cdot \mathcal{V}^{>-1}$  generated by the canonical extension \cite{overview}, and the lowest piece $F_p M$ is equal to $j_* F_p \mathcal{V} \cap \mathcal{V}^{>-1}$ \cite[Exercise 11.2]{overview}, where $j: X\backslash D\rightarrow X$ is the inclusion. 
% Let $-1 < \beta_{ik}<\cdots<\beta_{i2}< \beta_{i1} \leq 0$ be eigenvalues of residues of the canonical extension $\tilde{\V}^{>-1}$ along a component $D_i$ of the SNC divisor $D$. 

% Recall our notation and the main theorem:
% $$\Sigma_I= \{ (\beta_{i})_{i\in I}: \text{$\beta_{i}$ is an eigenvalue of residue along $D_i$ of the canonical extension of $\mathcal{V}$}\}.$$

% \maintheorem*

% \begin{theorem}
% 		Given a smooth projective curve $C$, and a map $f: C\rightarrow X$, 
% 		\begin{itemize}
% 			\item[(a)] If $\IM(C)\nsubseteq D$, then $f^* F_p M$ is a nef vector bundle on $C$. 

% 			\item[(b)] If $\IM(C)\subset D_{(k)}\defeq D_1\cap \cdots \cap D_k$, but the image of $C$ doesn't lie in any $(k+1)$ intersections of boundary divisors, and $f^* F_p M\rightarrow P$ is a quotient line bundle, then 
% 			$$\deg P \geq \min_{\text{over all possible }j_i} \big\{ \sum_{i=1}^k (- \beta_{ij_i}\cdot \deg f^* \strucsheaf_X(D_i)) \big\}.$$
	
% 		\end{itemize}

% 		\end{theorem}

% To illustrate the idea, we first prove our theorem in the surface case. In this section, we give a lower bound on line bundle quotients of lowest piece of the Hodge module when restricted to curves, in terms of local monodromy and intersection numbers with boundary divisors. Our formula exactly recovers Kawamata's result on nefness when the monodromy is unipotent. Combined with the following example in \cite[Example 7.9]{fujinovmhs}, this shows that our estimate is optimal in the surface case.
	
\begin{example}[\cite{fujinovmhs}]
	Consider the projective bundle $X=\proj(\strucsheaf\oplus \strucsheaf(2))$ on $\proj^1$, and $\pi:X\rightarrow \proj^1$. Consider the following two sections of $\pi$. $E$ corresponds to the split quotient $\strucsheaf_{\proj^1}\oplus\strucsheaf_{\proj^1}(2)\rightarrow \strucsheaf_{\proj^1}$, and $G$ corresponds to the split quotient $\strucsheaf_{\proj^1}\oplus\strucsheaf_{\proj^1}(2)\rightarrow \strucsheaf_{\proj^1}(2)$. Denote a fiber of $\pi$ by $F$. A simple computation shows $E^2=-2$, $G^2=2$, $E\cdot F=G\cdot F=1$, $E\cdot G=0$, and $2(E+F) \sim E+G$ are linearly equivalent. We define $L\defeq \strucsheaf_X(E+F)$, then there is a degree two cover $f:\tilde{X}\rightarrow X$, ramified along $E+G$. This cover induces a VHS $\mathcal{V}$ on $X\backslash (E\cup G)$. Its intermediate extension Hodge module $M$ has its lowest piece $F_p M=f_*\omega_{\tilde{X}/X}$. Since $f_*\omega_{\tilde{X}/X} = \strucsheaf_X\oplus L$, $\deg(L|_E)=-1$, $F_p M$ is not nef on $X$. 
\end{example}
	The monodromy matrices along $E$ and $G$ are
	$ \begin{pmatrix}
	0 & 1 \\
	1 & 0 
	\end{pmatrix}  $,
	whose eigenvalues are $1,-1$. A logarithm of it divided by $2\pi i$ is
	$ \begin{pmatrix}
	0 & 0 \\
	0 & -\frac{1}{2} 
	\end{pmatrix}  $.  Then the $V$-filtration of $M$ along $E$ or $G$ has indices jumping at $0$ and $-\frac{1}{2}$. Since $E^2=-2$, $\beta_1=0$, $\beta_2=-\frac{1}{2}$, we have
	$$\min\{-\beta_1\cdot E^2,   -\beta_2\cdot E^2 \}=-1.$$ 
	This shows our estimates explain the example in \cite{fujinovmhs}. In a later section more examples of this type will be given, and show that our lower bound can be achieved in many geometric situations.

	One could try to argue as previous classical works using curvature properties of Hodge metric on $F_p M$, which is a singular Hermitian metric with semi-positive curvature. However, the problem comes from the possible appearance of nonzero Lelong numbers of the metric along the boundary divisor. In the setting of \cite{semipo}, when the VHS is unipotent, the Lelong numbers are zero everywhere, then by the current approximation theorem of Demailly, it can be shown that the lowest piece is nef. The metric growth of canonical extension of $\mathcal{V}$ is controlled by the indices of the $V$-filtration, hence the Lelong numbers of the Hodge metric on $F_p M$ are controlled by these indices.  In the general case, it seems that the jumping indices of the $V$-filtration, or the nonzero Lelong numbers of the Hodge metric are responsible for the failure of nefness. Our theorem may therefore also be proved using analytic arguments. Nonetheless, our method is algebraic, and suggests the usefulness of twisted Hodge modules when specializing to divisors.

\section{Bounding failure of nefness of the lowest piece}
\indent \indent In this section, as before, we state the results in terms of left $\D$-modules, and during the proofs we use right $\D$-modules.  This is because the twistings for twisted Hodge modules are more natural for right $\D$-modules. Throughout the remainder of the paper, $M$ will always mean a left $\D$-module, and the corresponding right $\D$-module $M\otimes_{\strucsheaf_X}K_X$ will be denoted by $N$.  $X$ will be a compact complex manifold, $D=\sum_i D_i$ is an SNC divisor on $X$, $\mathcal{V} $ is a polarized variation of Hodge structures on $X\backslash D$. For convenience, both $\gr^V M$ and $\gr^W\gr^V M$ for a Hodge module $M$ are called nearby cycles of $M$. This won't cause confusion in what follows.

Our main idea is to iteratively perform nearby cycles $\gr^{W_i}\gr^{V_i} M$ several times to restrict Hodge modules to intersections of divisors. However, there's an issue. For $i\neq j$, we can first take $M'\defeq \gr^{W_i}\gr^{V_i}M$, get a twisted Hodge module on $X$, and then apply a $\gr^{W'_j}\gr^{V'_j}$ nearby cycle along $D_j$, where $i\neq j$ and $V_i$ means the $V$-filtration along $D_i$, and $W_i$ is the associated weight filtration. Here $V'_j$ is the $V$-filtration for $M'$ along $D_j$, but it's not clear whether this should be the projection of the original $V$-filtration of $M$ along $D_j$. In our simple normal crossing case, this is true, by the nilpotent orbit theorem \cite{cks}, the process of taking nearby cycles doesn't depend on the order, and all the jumping indices for $V$-filtrations are just residues of the canonical extension along the divisors.

We will need a technical lemma below. It basically says, for our purpose, we may assume the twisted Hodge module $\gr^{W_1}\gr^{V_1} M$ is of strict support, i.e., $\gr^{W_1}\gr^{V_1} M$ is the intermediate extension of the (twisted) VHS on the open part. Given a VHS on $(\Delta^*)^n$, we can extend it as a Hodge module $M$. Consider the $V$-filtration of $M$ along a coordinate hyperplane, say $D_1=(z_1=0)$. The following lemma says, if $F_p M$ is the lowest piece of $M$, then $F_p \gr^{W_1}\gr^{V_1} M$ is torsion free on $D_1$. Even if $\gr^{W_1}\gr^{V_1} M$ has direct summands supported on lower dimensional subvarieties of $D_1$, the twisted Hodge modules supported on lower dimensional components do not contribute to the lowest piece, which means $F_p \gr^{W_1}\gr^{V_1} M$ only comes from the VHS on $D_1^*=(\Delta^*)^{n-1}$. $\gr^{V_1}M$ can be viewed (locally) as a polarizable Hodge-Lefschetz Hodge module \cite[Definition 3.6]{davis2024archimedeanzetafunctionssingularities}, which is torsion-free (as a quasi-coherent sheaf) by well-known results in \cite{schmid} and \cite{cks}.

%\textcolor{red}{I suspect something stronger is true, i.e., $\gr^{W_1}\gr^{V_1} M$ has pure support on $D_1$.}

\begin{lemma}\label{lem:lowestpiecenearbycycle}
	Let $(M,N,W_{\sbullet} = W(N))$ be a polarizable Hodge-Lefschetz module in the sense of Definition 3.6 in \cite{davis2024archimedeanzetafunctionssingularities}, whose underlying quasi-coherent sheaf is torsion-free on a smooth variety $Y$. Assume that $F_p M\neq 0$ and $F_{p-1}M=0$, then the lowest pieces of Hodge modules $F_p\gr^W_r M$ for $r\geq 0$ are torsion-free on $Y$, and $F_p \gr^W_r M =0$ for $r<0$.
\end{lemma}
\begin{proof}
	Let $m\in F_p M \cap ( W_l M\backslash W_{l-1} M)$ be a local section such that $f\cdot [m]=0\in \gr^W_l M$ for some nonzero local holomorphic function $f$. By \cite[Lemma 3.7]{davis2024archimedeanzetafunctionssingularities}, $l$ is the smallest integer such that $N^{l+1}m=0\in M$. And we know $f\cdot m\in W_{l-h}\backslash W_{l-h-1}$ for some $h\geq 1$ since $f\cdot[m]=0$. Also $f\cdot m\in F_p M$, then by that lemma again we have $N^{l-h+1} fm = f(N^{l-h+1}m)=0$. Since $M$ is assumed to be torsion-free, $N^{l-h+1}m=0$, which contradicts the minimality of $l$. Hence $F_p\gr^W_l M$ is torsion-free. 
\end{proof}

We prove the first part of the main theorem, which is the case when $C$ is not inside the boundary divisor. Also note that here we don't actually need to assume that $D$ is SNC, and the price is $F_p M$ may not be locally free.
% \begin{theorem}
% 	Let $X$ be a compact complex manifold, $D$ is a divisor on $X$, $\mathcal{V}$ is a VHS on $X\backslash D$, and $M$ is the extension of $\mathcal{V}$ as a Hodge module. Given a map $f:C\rightarrow X$ where $C$ is a smooth projective curve, if $\IM(C)\nsubseteq D$, then $f^* F_p M$ is a nef coherent sheaf on $C$. 
% \end{theorem}
\begin{proof}[Proof of first part of Theorem \ref{thm:main} ]
	Since the image of $C$ is not contained in $D$, we can first pullback the generic VHS to a Zariski open set of $C$, and extend it to a Hodge module $M_C$ on $C$. By functoriality of Hodge modules, as long as the image of $C$ is not contained in $D$, there exists a map $F_p M_C \rightarrow f^* F_p M$ which is generically an isomorphism \cite[Lemma 2.6]{bakker2023hodgetheoreticproofhwangstheorem}. 
	
	Since $F_p M_C$ is a nef vector bundle by \theoremref{thm:curvetwistingtheorem}, $f^* F_p M$ is also a nef coherent sheaf on $C$. The same argument works when $D$ is not simple normal crossing, so when $D$ is not SNC, $f^* F_p M$ is still a nef coherent sheaf on $C$.
\end{proof}

Now let $X,D,\mathcal{V}$ be as at the beginning of this section, i.e., $X$ will be a compact complex manifold, $D=\sum_i D_i$ is an SNC divisor on $X$, $\mathcal{V} $ is a polarized variation of Hodge structures on $X\backslash D$. For a fixed $i$, let $-1 <\beta_{ij}\leq 0$ be eigenvalues of residues of the canonical extension $V^{>-1}$ along $D_i$. 

For completeness and convenience, we include the following lemma. It says when a coherent sheaf on a curve is filtered, and we want to give a lower bound on the degrees of its line bundle quotients, it suffices to give such lower bounds on the graded objects of the filtration. To keep the statement of the lemma concise, we denote by $d(\F)$ the smallest degree of quotient line bundles of a coherent sheaf $\F$ on a curve. 
	\begin{lemma}\label{bounding degrees of filtrations}
		Let $\F$ be a coherent sheaf on a smooth projective curve, equipped with a finite increasing filtration $L_{\sbullet} \F$. Then 
		$$d(\F)\geq \min_i \{d(\gr^L_i \F)\}.$$	
	\end{lemma}
	\begin{proof}
		Assume $0=L_0 \F \subset L_1\F\subset \cdots \subset L_k \F=\F $. It suffices to show $d(\F)\geq \min\{d(\gr^L_k \F),d(L_{k-1}\F)\}$. Then an easy induction will complete the proof of the lemma.
		
		Let $\F\rightarrow P$ be a quotient line bundle. If the composition of the morphisms $L_{k-1}\F\rightarrow \F\rightarrow P$ is nonzero, then $\deg P\geq d(L_{k-1} \F)$; otherwise the morphism $\F\rightarrow P$ factors through $\F\rightarrow \gr^L_k \F\rightarrow P$, so $P$ is also a quotient of $\gr^L_k \F$, and we have $\deg P\geq d(\gr^L_k \F)$. In either case, we get $d(\F)\geq \min\{d(\gr^L_k \F),d(L_{k-1}\F)\}$.
		
	\end{proof}

% We now prove the main result \theoremref{thm:main}, recall that
% $$\Sigma_I= \{ (\beta_{i})_{i\in I}: \text{$\beta_{i}$ is an eigenvalue of residue along $D_i$ of the canonical extension of $\mathcal{V}$}\}.$$
% \maintheorem*
% \begin{theorem}
% 	Let $C$ be a smooth projective curve. Let $f: C\rightarrow X$, such that $\IM(C)\subset D_{(k)}\defeq D_1\cap \cdots \cap D_k$, but the image of $C$ doesn't lie in any $(k+1)$ intersections of boundary divisors, and $f^* F_p M\rightarrow P$ is a quotient line bundle, then 
% 	$$\deg P \geq \min_{\text{over all possible }j_i} \big\{ \sum_{i=1}^k (- \beta_{ij_i}\cdot \deg f^* \strucsheaf_X(D_i)) \big\}.$$
	 
% 	 %When the monodromies are unipotent, we have all $\beta_{rs}$ equal to zero, this implies $F_p M$ is a nef vector bundle on $X$. Therefore we have recovered the semi-positivity theorem of Kawamata \cite{Kawamata1981}, see also \cite{fujinovmhs}, \cite{MR3167580}, \cite{semipo}.
% \end{theorem}

\begin{proof}[Proof of Theorem \ref{thm:main}]
	If $\IM(C)\subset D$, to illustrate the idea, assume $f: C\rightarrow D_1\hookrightarrow X$, and the image of $C$ doesn't lie in the intersections of boundary divisors. Remember $N=M\otimes K_X$ is the corresponding right $\D$-module. Consider the $V$-filtration of $N$ along $D_1$. As explained in section 2,  $F_p N|_{D_1} $ is filtered by coherent subsheaves 
		$$\frac{F_pV_{\alpha_1}N}{F_p V_{<-1}N},\cdots,\frac{F_pV_{\alpha_k}N}{F_p V_{<-1}N},$$
		where $-1\leq \alpha_1<\cdots<\alpha_k<0$ are related to the monodromy along $D_1$; in fact, $\beta_i= -(\alpha_{i}+1) \in (-1,0]$ are exactly the eigenvalues of residues of canonical extension of $\V$ along $D_1$. The graded quotients of this filtration by subsheaves are $F_p \text{gr}^V_{\alpha_i} N$, each of which is the pushforward of a coherent sheaf from $D_1$ to $X$. By \cite[Proposition 3.12]{schnell2024highermultiplierideals}, we know $\gr^W_l \gr^V_\alpha N$ is an $\alpha L$-twisted Hodge module on $X$ supported on $D_1$, for $\alpha\in [-1,0)$, $l\in \Z$, where $L=\strucsheaf_X(D_1)$. Now by Lemma \ref{bounding degrees of filtrations}, it remains to understand the degree lower bounds of quotient line bundles of the lowest piece of twisted Hodge modules $F_p \gr^W_l\gr^{V}_{\alpha_i} N\otimes K_{X}^{-1}$ on $D_1$, since $F_p \gr^W_l\gr^{V}_{\alpha_i} N\otimes K_{X}^{-1}$ are subquotients of $F_p M|_{D_1}$.
		
		By \theoremref{thm:curvetwistingtheorem}, every line bundle quotient $P_{\alpha_i,l}$ of $$f^* (F_p \gr^W_l\gr^V_{\alpha_i} N \otimes K_{X}^{-1}) = f^* (F_p \gr^W_l\gr^V_{\alpha_i} N\otimes K_{D_1}^{-1}\otimes (K_{D_1}\otimes K_{X}^{-1}|_{D_1}) )$$ has $\deg P_{\alpha_i, l} - \deg f^*(K_{D_1}\otimes K_{X}^{-1}|_{D_1}) \geq \alpha_i \cdot \deg f^* L$, hence
		$$\deg P_{\alpha_i,l}\geq (1+\alpha_i)\cdot \deg f^*\strucsheaf_X(D_1) = -\beta_i \deg f^* \strucsheaf_X(D_1).$$
		Here the appearance of canonical bundles of $X$ and $D_1$ is due to left-right conversions of $\D$-modules, since our convention for taking twisted Hodge modules is using right $\D$-modules \cite[Remark 2.3]{schnell2024highermultiplierideals}, but the object we are interested in is the lowest piece of the Hodge filtration of the left $\D$-module.
				
		Now we return to our original object, $f^* F_p M$. It's filtered by some coherent subsheaves, and the quotient line bundles of the subquotients have degree lower bounds in terms of $\deg f^*\strucsheaf_X(D_1)$ and monodromy along $D_1$. Then by Lemma \ref{bounding degrees of filtrations}, the degree of the quotient line bundles $f^* F_p M\rightarrow P$ is bounded below by the minimum of all the lower bounds, if given a map $f: C\rightarrow D_1\hookrightarrow X$,
		$$\deg P\geq \min_i \{-\beta_i \cdot \deg f^*\strucsheaf_X(D_1)\}.$$

	In the general case, if $\IM(C)\subset D_{(k)}\defeq D_1\cap \cdots \cap D_k$, but the image of $C$ doesn't lie in any $(k+1)$ intersections of boundary divisors, then
	we take nearby cycles along $D_1$ up to $D_k$, and the object we get is an
	$$(\alpha_{1}\cdot \strucsheaf_X(D_1)|_{D_{(k)}},\cdots,\alpha_{k}\cdot \strucsheaf_X(D_k)|_{D_{(k)}})\text{-twisted Hodge module } N_k,$$
	where $\alpha_i\in [-1,0)$ is a jumping index of the $V$-filtration of $N_k$ along $D_i$.
	By assumption $\IM(C)$ is not contained in the singular locus of this (twisted) Hodge module, we can first pullback the (twisted) VHS to a Zariski open set of $C$, then extend it as a multiply twisted Hodge module $N'$ on $C$. We may assume $N_k$ is of strict support on $D_1\cap \cdots \cap D_k$, by our discussion before \lemmaref{lem:lowestpiecenearbycycle}. The following lemma which is an analogue of \cite[Lemma 2.6]{bakker2023hodgetheoreticproofhwangstheorem} in the twisted case is needed.

	\begin{lemma}\label{lem:twisted VHS generic pullback lowest piece}
	In the setting of the preceding paragraph, there exists a twisted Hodge module $N'$ on $C$, whose lowest piece of Hodge filtration is denoted by $F_p N'$, such that there exists a map $F_p N'\otimes K_C^{-1}\rightarrow f^* (F_p N_k\otimes K_{D_{(k)}}^{-1} )$ which is generically an isomorphism. 
	\end{lemma}
	\begin{proof}
		Let $L_i = \strucsheaf_X(D_i)|_{D_{(k)}}$. Recall that the $(\alpha_{1}L_1,\cdots,\alpha_{k} L_k)$-twisted Hodge module is a quasi-coherent sheaf $N_k$ admitting an action by the ring $\D_{D_{(k)},\alpha_1 L_1,\cdots, \alpha_k L_k}$. Locally on a small open set $U$ of $D_{(k)}$, after choosing a trivialization of $L_1,\cdots,L_k$, we have an isomorphism 
		$$\D_{D_{(k)},\alpha_1 L_1,\cdots, \alpha_k L_k}|_{U} \cong \D_U[t_1\partial_{t_1},\cdots,t_k\partial_{t_k}]/(t_i\partial_{t_{i}} -\alpha_i) \cong \D_U.$$
		By construction and the assumption on $N_k$, $N_k$ is a locally free $\strucsheaf_{D_{(k)}}$ module on $D_{(k)}\backslash D_{(k+1)}$. Set $L_i'=f^* L_i$ on $C$. Let $C^{\circ} = C\backslash f^{-1}(D_{(k+1)})$, $f^\circ: C^\circ\rightarrow D_{(k)}$, and there's a natural map 
		$$f_*: \D_{C,\alpha_1 L_1',\cdots,\alpha_k L_k'}\rightarrow f^*\D_{D_{(k)},\alpha_1 L_1,\cdots, \alpha_k L_k}.$$
		It turns out that $f^{\circ *} N_k$ is a $ \D_{C,\alpha_1 L_1',\cdots,\alpha_k L_k'}$-module on $C^\circ$, which is a twisted VHS on $C^\circ$ (with the pullback of the Hodge filtration, see \cite[Definition 2.10]{schnell2024highermultiplierideals} for the definition of twisted VHS). This is because being a twisted VHS (or Hodge module) is a local property \cite[Definition 2.7]{schnell2024highermultiplierideals}, $N_k$ being a twisted VHS on $D_{(k)}\backslash D_{(k+1)}$ implies $f^{\circ *} N_k$ is a twisted VHS on $C^\circ$. By the structure theorem of twisted Hodge modules \cite[Theorem 2.13]{schnell2024highermultiplierideals}, the twisted VHS $f^{\circ *} N_k$ on $C^\circ$ induces an $(\alpha_1 L_1',\cdots,\alpha_k L_k')$-twisted Hodge module $N'$ on $C$. This procedure is called generic pullback in \cite[after Theorem 2.13]{schnell2024highermultiplierideals}. 

		Now we show there exists a natural map $F_p N'\otimes K_C^{-1}\rightarrow f^* (F_p N_k\otimes K_{D_{(k)}}^{-1} )$ which is generically an isomorphism. The appearance of various canonical bundles here is because of left-right $\D$-module conversion, since we are proving the theorem using nearby cycles for right $\D$-modules, but the statement of the main theorem is about the lowest piece of the left $\D$-module. This map is analogous to the map in \cite[Lemma 2.6]{bakker2023hodgetheoreticproofhwangstheorem}, once the canonical bundles are interpreted correctly. Let $p\in C$ be a point. We can find trivializations of $L_i$ on a small Zariski open neighborhood $U$ of $f(p)$, which also give trivializations of $L_i'$ on a neighborhood $U'$ of $p$. Using the trivializations, we can view $N_k$ and $N'$ as untwisted Hodge modules on $U$ and $U'$. Then the natural map $F_p N'\otimes K_C^{-1}|_{U'} \rightarrow f^* (F_p N_k\otimes K_{D_{(k)}}^{-1} ) |_{U'}$ exists and is generically an isomorphism on $U'$ by \cite[Lemma 2.6]{bakker2023hodgetheoreticproofhwangstheorem}. These maps on open sets are defined using trivializations. Since they are natural, they glue together to a globally defined map $F_p N'\otimes K_C^{-1}\rightarrow f^* (F_p N_k\otimes K_{D_{(k)}}^{-1} )$, which is again generically an isomorphism.
		
		% It's then an easy manipulation of coordinates for ring of twisted differential operators that these maps defined using local trivializations glue together to give a globally defined map $F_p N'\otimes K_C^{-1}\rightarrow f^* (F_p N_k\otimes K_{D_{(k)}}^{-1} )$, which is generically an isomorphism.

	\end{proof}
	
	We have a map $F_p N'\otimes K_C^{-1}\rightarrow f^* (F_p N_k\otimes K_{D_{(k)}}^{-1} )$ which is generically an isomorphism. Let $P$ be a line bundle which is a quotient of $f^* (F_p N_k\otimes K_{X}^{-1} )$. The composition of maps $$F_p N'\otimes K_C^{-1} \rightarrow f^* (F_p N_k\otimes K_{X}^{-1} ) \otimes f^* (K_X\otimes K_{D_{(k)}}^{-1})\rightarrow P \otimes f^* (K_X\otimes K_{D_{(k)}}^{-1})$$ is then nonzero. By an application of \theoremref{thm:curvetwistingtheorem} in the multiply twisted case, $\deg P$ is greater than or equal to the degree of $$\sum \alpha_i f^*\strucsheaf_X(D_i) + f^*(K_{D_{(k)}} \otimes K_X^{-1}) = \sum (1+\alpha_i)f^* \strucsheaf_X(D_i).$$ This implies our theorem, because $(1+\alpha_i)=-\beta_i$.  
\end{proof}

Now we state two corollaries. The first one is originally due to Kawamata \cite{Kawamata1981}. His proof uses Cattani-Kaplan-Schmid's results about VHS, and our proof of the main theorem is essentially in the same spirit as Kawamata's proof, by specializing the lowest piece to the intersection of divisors. It's also proved in \cite{semipo} using asymptotics of the Hodge metric, and a current approximation theorem in complex analytic geometry. See also \cite{fujinovmhs} and \cite{MR3167580}. The second corollary is new. When $X$ is an abelian variety, it is also proved in \cite[Corollary 20.2]{pps2017}, using methods of generic vanishing theory. They actually proved something stronger in that case using special geometry of abelian varieties, that $F_p M$ has a smooth Hermitian metric with semi-positive curvature.

\begin{corollary}
	If $D$ is a normal crossing divisor, and the monodromies along each component of $D$ are unipotent, then $F_p M$ is a nef vector bundle on $X$.
\end{corollary}
\begin{proof}
	When $D$ is simple normal crossing, because all $\beta_{ij}=0$ in the case of unipotent monodromy, then the main theorem implies all quotient line bundles have non-negative degrees.
	
	If the divisor is only assumed to be normal crossing, then each component $D_i$ could have self-intersections. We blow up all the intersections of divisors and get $\mu: \tilde{X}\rightarrow X$, to make $\mu^* D$ have simple normal crossing support. Pullback the canonical extension $V$ on $X$ to $\tilde{X}$, call it $\tilde{V}$, and $\tilde{V}$ is a locally free sheaf with a logarithmic connection along $\mu^{-1}(D)$, extending the VHS on $\tilde{X}-\mu^{-1}(D) $. Since the VHS on $X-D$ is unipotent, the residues of the canonical extension are nilpotent. By a local computation of blowups, $\tilde{V}$ also has nilpotent residues. By uniqueness \cite[Theorem 5.2.17]{htt}, this implies $\tilde{V}$ is the canonical extension on $\tilde{X}$. Denote the corresponding Hodge module on $\tilde{X}$ by $\tilde{M}$. We have $F_p \tilde{M} = F_p \tilde{V} $ by \cite[Exercise 11.2]{overview}. Then
	$$F_p \tilde{M} = F_p \tilde{V}  = \mu^* F_p V = \mu^* F_p M.$$
	By the first paragraph, $F_p \tilde{M}$ is nef, then $F_p M$ is also nef.
\end{proof}

Our main theorem suggests that nefness is also related to the positivity of boundary divisors. In the surface case we can also relax the condition of simple normal crossing boundary.

\begin{corollary}
	\begin{itemize}
		\item[(a)] If each component of $D$ is a nef divisor, and $D$ is simple normal crossing, then $F_p M$ is a nef vector bundle on $X$. 
		
		\item[(b)] If $X$ is a surface, and each $D_i$ is a nef divisor (could be singular and $D$ needs not be SNC), then $F_p M$ is a nef coherent sheaf on $X$.
	\end{itemize}
	
\end{corollary}

\begin{proof}
	For (a), this is because $-\beta_{ij}\geq 0$, and each $\strucsheaf_X(D_i)$ is a nef line bundle. For (b), the same conclusion of the main theorem is still true for surfaces, with possibly non-SNC boundary divisor $D$, since taking nearby cycles along a singular curve on a surface can only have extra components of twisted Hodge modules at the singular points of the curve, they won't produce quotient line bundles with negative degrees.
\end{proof}

%\begin{corollary}
%	If $D$ is only a normal crossing divisor, and each component is nef, is $F_p M$ nef????????
%\end{corollary}
%\begin{proof}
%	I think this is definitely not true. Need to find an example, maybe base is $\proj^3$, outside a hypersurface having a transversal self-intersection, it's nef since $\proj^3$ has Picard rank one. There is a VHS, after blow up the intersection, the lowest piece can have a negative quotient on the exceptional divisor. 
%\end{proof}

We prove the last proposition in the introduction. Note that here we don't assume $D$ is simple normal crossing. 

\begin{proposition}
	Let $X$ be a smooth projective variety of dimension $n$, and let $M$ be a Hodge module induced by a variation of Hodge structures on a Zariski open subset of $X$. Let $f:C\rightarrow X$ be a map from a smooth projective curve $C$. If the image of $C$ is a smooth complete intersection curve of ample divisors, $f^*F_p M$ is a nef coherent sheaf on $C$. 
\end{proposition}
\begin{proof}
	Suppose $C$ is the transversal intersection of ample divisors $H_1,\cdots,H_{n-1}$. Let $F_p M$ be the lowest piece of the Hodge module we are considering. $F_p M|_C$ is filtered by some coherent subsheaves, whose subquotients are exactly equal to the lowest piece of some twisted Hodge modules supported on $C$ (produced by taking nearby cycles). By \cref{bounding degrees of filtrations}, it suffices to give degree lower bounds for these lowest pieces of twisted Hodge modules restricted to $C$.   
	
	We successively take nearby cycles along the ample divisors $H_1,\cdots,H_{n-1}$. This produces an $(\alpha_1  H_1,\cdots,\alpha_{n-1} H_{n-1})$-twisted right Hodge module $N$ on $X$, with support contained in $C$, where $\alpha_i\in [-1,0)$ are real numbers. We may assume $N$ has strict support $C$, since if the support is a point then it won't have quotient line bundles. Let $F_p N\otimes K_X^{-1}|_C\rightarrow P$ be a quotient line bundle on $C$. The factor $K_X^{-1}$ shows up since $N$ is a right twisted Hodge module. By our previous result \theoremref{thm:curvetwistingtheorem}, we have 
	$$F_p N\otimes K_{C}^{-1}\rightarrow P\otimes (K_{C}^{-1}\otimes K_X|_C), \quad \deg (P\otimes (K_{C}^{-1}\otimes K_X|_C)) \geq \sum_{i}\alpha_i\cdot \deg H_i|_{C}.$$
	Since $K_C\otimes K_{X}^{-1}|_C \cong \otimes_i \strucsheaf_X(H_i)|_C$, we have $$\deg P \geq \sum_i (1+\alpha_i) \cdot \deg\strucsheaf_X(H_i)|_C\geq 0.$$ 
	For curves mapping to a smooth complete intersection curve, the proof is similar, just pullback the generic twisted VHS on the complete intersection curve, and use \theoremref{thm:curvetwistingtheorem}.
\end{proof}

\begin{remark}
    % The above proposition is really a local statement around the curve. If normal bundle of $C$ is nef and splits into line bundles, and $T_X|_C \cong T_C\oplus N_{C/X}$ (hence it has a tubular neighborhood in $X$ isomorphic to a neighborhood of the zero section in its normal bundle) , then by the same proof as above $F_p M|_C$ is nef. 

In the preceding proposition, we do not need to keep track of the order in which nearby cycles are taken or whether they yield a twisted Hodge module with strict support, since in the end the coefficients of the twistings turn out to be nonnegative. However, during the proof of the main theorem, it's important to know the lowest piece of nearby cycles comes from a twisted Hodge module with strict support, since we want to know explicitly the coefficients of the lower bound, which are eigenvalues of residues of Deligne canonical extension.
\end{remark}

\section{Examples}

We show some examples where our lower bounds are achieved.
\begin{example}
	The example in \cite{fujinovmhs} also works in dimensions greater than two. Take $\pi:X\defeq\proj(\strucsheaf_{\proj^n}\oplus\strucsheaf_{\proj^n}(2)) \rightarrow \proj^n$. We have two sections $E,G$. $E$ corresponds to the quotient $\strucsheaf_{\proj^n}\oplus\strucsheaf_{\proj^n}(2)\rightarrow \strucsheaf_{\proj^n}$, and $G$ corresponds to $\strucsheaf_{\proj^n}\oplus\strucsheaf_{\proj^n}(2)\rightarrow \strucsheaf_{\proj^n}(2)$. A simple computation shows $\strucsheaf_X(E+G)$ is isomorphic to $L^{\otimes 2}\defeq 2(\strucsheaf_X(1) - \pi^* \strucsheaf_{\proj^n}(1))$. Let $f:\tilde{X}\rightarrow X$ be the branched cover of $X$ of degree $2$, branched over $E,G$. We get $f_*\omega_{\tilde{X}/X}\cong \strucsheaf_X\oplus L$. Since $L|_{E}\cong \strucsheaf_{\proj^n}(-1)$, the degree of $L$ restricted to any line in $E\cong \proj^n$ is equal to $-1$. This shows our estimate is sharp in some higher dimensional examples.
\end{example}

In the above example, two hypersurfaces are disjoint. Next we present an example of a 3-fold $X$, where the lower bound can still be achieved on a curve which is the intersection of two smooth surfaces in $X$. The same idea can be used to produce more examples in higher dimensions, indicating our lower bound can be achieved on intersections of boundary divisors in any dimension.

\begin{example}
	Let $Y$ be a smooth projective 3-fold, and let $D_1, D_2$ be smooth hypersurfaces in $Y$ intersecting transversely along a curve $C$. Suppose $\deg \strucsheaf_{Y}(D_1)|_C <0$, $\deg \strucsheaf_{Y}(D_2)|_C <0$, and there exists a line bundle $L$ on $Y$ such that $L^N\cong  \strucsheaf_{Y}(D_1+D_2)$ for some positive integer $N>1$. 
	
	Consider a branched cover of $Y$, by taking the $N$-th roots of the section of $L^N$ vanishing along $D_1+D_2$. Resolving the singularities of this branched cover, we get another smooth projective 3-fold $X$, with a map $g:X\rightarrow Y$ étale over $Y-(D_1+D_2)$, and $$g_*\omega_{X/Y}\cong \oplus_{i=0}^{N-1} L^i.$$ 
	The eigenvalues of residues of the canonical extension of the VHS induced by $g$ on $Y-(D_1+D_2)$ along $D_1,D_2$ are $-\frac{N-1}{N},-\frac{N-2}{N},\cdots,\frac{-1}{N},0$. By the main theorem, any line bundle quotient of $(g_*\omega_{X/Y})|_C$ has degree greater than or equal to
	$$\min_{a,b\in \{0,1,\cdots, N-1\}} \big\{ \frac{a}{N}\cdot \deg \strucsheaf_Y(D_1)|_C + \frac{b}{N}\cdot \deg \strucsheaf_Y(D_2)|_C \}.$$
	By assumption, since $\deg \strucsheaf_{Y}(D_1)|_C <0$ and $\deg \strucsheaf_{Y}(D_2)|_C <0$, the above expression is equal to $$\frac{N-1}{N}\cdot \deg \strucsheaf_Y(D_1+D_2)|_C = \deg L^{N-1}|_{C}<0.$$
	This shows that the quotient line bundle $L^{N-1}|_C$ of $(g_*\omega_{X/Y}|_C)$ achieves the equality case of the lower bound, and $g_*\omega_{X/Y}$ is not nef.
	
\end{example}

\bibliographystyle{alpha}
\bibliography{references.bib}

\end{document}